\documentclass[12pt,fleqn]{article}
\usepackage{xcolor}

\usepackage{amscd,amsmath,amssymb,amsthm}
\usepackage{geometry} 
\usepackage[pdfpagemode=UseNone,pdfstartview=FitH]{hyperref}
\usepackage{mathrsfs}

\newcommand{\A}{{\mathcal A}}
\newcommand{\Ap}[1][]{A_p #1}
\newcommand{\abs}[1]{\left|#1\right|}
\newcommand{\Bp}[1][]{B_p #1}
\newcommand{\bdry}[1]{\partial #1}
\newcommand{\C}[1]{$(\text{C})_{#1}$}
\newcommand{\comp}{\circ}
\newcommand{\dnorm}[2][]{\|#2\|_{#1}^\ast}
\newcommand{\dualp}[3][]{\left(#2,#3\right)_{#1}}
\newcommand{\E}{\mathscr E}
\newcommand{\eps}{\varepsilon}
\newcommand{\F}{{\mathcal F}}
\newcommand{\goodchi}{\protect\raisebox{2pt}{$\chi$}}
\newcommand{\half}{\frac{1}{2}}
\newcommand{\I}{\mathscr I}
\newcommand{\id}[1][]{id_{#1}}
\newcommand{\incl}{\subset}
\newcommand{\isom}{\approx}
\newcommand{\M}{\mathcal M}
\newcommand{\N}{\mathbb N}
\newcommand{\norm}[2][]{\left\|#2\right\|_{#1}}
\renewcommand{\o}{\text{o}}
\newcommand{\PS}[1]{$(\text{PS})_{#1}$}
\newcommand{\pnorm}[2][]{\if #1'' \left|#2\right|_p \else \left|#2\right|_{#1} \fi}
\newcommand{\qnorm}[2][]{\if #1'' \left|#2\right|_q \else \left|#2\right|_{#1} \fi}
\newcommand{\R}{\mathbb R}
\newcommand{\RP}{\R \text{P}}
\newcommand{\restr}[2]{\left.#1\right|_{#2}}
\newcommand{\seq}[1]{\left(#1\right)}
\newcommand{\set}[1]{\left\{#1\right\}}
\newcommand{\vol}[1]{\left|#1\right|}
\newcommand{\wto}{\rightharpoonup}
\newcommand{\Z}{\mathbb Z}

\DeclareMathOperator{\dist}{dist}

\newenvironment{enumroman}{\begin{enumerate}
\renewcommand{\theenumi}{$(\roman{enumi})$}
\renewcommand{\labelenumi}{$(\roman{enumi})$}}{\end{enumerate}}

\newenvironment{properties}[1]{\begin{enumerate}
\renewcommand{\theenumi}{$(#1_\arabic{enumi})$}
\renewcommand{\labelenumi}{$(#1_\arabic{enumi})$}}{\end{enumerate}}

\newtheorem{corollary}{Corollary}[section]
\newtheorem{lemma}[corollary]{Lemma}
\newtheorem{proposition}[corollary]{Proposition}
\newtheorem{theorem}[corollary]{Theorem}

\theoremstyle{definition}
\newtheorem{definition}[corollary]{Definition}

\theoremstyle{remark}
\newtheorem{example}[corollary]{Example}
\newtheorem{remark}[corollary]{Remark}

\numberwithin{equation}{section}

\title{\bf Nonlocal Dirichlet problems involving the Logarithmic $p$-Laplacian \thanks{{\em MSC2020:} Primary 35J60, Secondary 35A15, 35B33, 58E05.
\newline \indent\; {\em Key Words and Phrases: Logarithmic $p$-Laplacian, Eigenvalue problem, $p$-logarithmic Sobolev inequality, Existence results, Morse theory, Fadell-Rabinowitz cohomological index.}}}
\author{\bf Rakesh Arora\\
Department of Mathematical Sciences\\
Indian Institute of Technology Varanasi (IIT-BHU)\\
Uttar Pradesh 221005, India\\
\em rakesh.mat@iitbhu.ac.in\\
[\medskipamount]
\bf Hichem Hajaiej\\
Department of Mathematics\\
California State University\\
Los Angeles, CA 90032, USA\\
\em hichem.hajaiej@gmail.com\\
[\medskipamount]
\bf Kanishka Perera\\
Department of Mathematics\\
Florida Institute of Technology\\
Melbourne, FL 32901, USA\\
\em kperera@fit.edu}
\date{}

\begin{document}

\maketitle

\begin{abstract}
In this work, we show the existence of an unbounded sequence of minimax eigenvalues for the logarithmic $p$-Laplacian via the $\Z_2$-cohomological index of Fadell and Rabinowitz. As an application of these minimax eigenvalues and $p$-logarithmic Sobolev inequality proved in \cite{AroGiaHajVai-2025}, we prove new existence results for nonlocal Dirichlet problems involving logarithmic $p$-Laplacian and nonlinearities with $p$-superlinear and subcritical growth at infinity.
\end{abstract}

\section{Introduction}

For $p \in (1,\infty)$, $N \geq 1$, and $s \in (0,1)$, the fractional $p$-Laplacian $(- \Delta_p)^s$ is the nonlinear nonlocal operator defined on suitably regular functions $u$ by
\[
(- \Delta_p)^s u(x) = C_{N,s,p} \text{ P.V.} \int_{\R^N} \frac{|u(x) - u(y)|^{p-2}\, (u(x) - u(y))}{|x - y|^{N+sp}}\, dy, \quad x \in \R^N,
\]
where $C_{N,s,p}$ is a normalizing constant chosen such that the limits
\[
\lim_{s \to 0^+}\, (- \Delta_p)^s u(x) = |u(x)|^{p-2}\, u(x), \qquad \lim_{s \to 1^-}\, (- \Delta_p)^s u(x) = - \Delta_p u(x)
\]
hold. This operator can be viewed as a nonlocal analogue of the
classical $p$-Laplace operator. Its analytical theory, covering existence, uniqueness, regularity, and qualitative properties of solutions, has been the focus of extensive research; see, for instance, \cite{Nezza-Palatucci-Valdinoci, Frank-Konig-Tang, Iannizzotto et al-ACV-2016,  Iannizzotto-Staicu-Vespri-2025}.

A particularly interesting direction within nonlocal analysis concerns operators with zero-order kernels, i.e., kernels behaving like $|x-y|^{-N}$ near singularity and having infinite range of interactions. These operators arise naturally both in mathematical models exhibiting borderline singular behavior and in a variety of applied contexts. Their study has produced a rich theoretical framework; see \cite{Chen-Weth-CPDE-2019, Feulefack-Jarohs, Frank-Konig-Tang, Foghem-Nodea-2025, Jarohs-Saldana-Weth-1, Temgoua-Weth} for
mathematical developments and \cite{Pellacci-Verzini, Sikic-Song-Vondracek, Sprekels-Valdinoci} for applications.

A seminal work in this direction was the introduction of the logarithmic Laplacian $L_{\Delta}$, a singular integro-differential
operator derived by Chen and Weth \cite{Chen-Weth-CPDE-2019}. For a smooth function $u$ and $x \in \mathbb{R}^N$, it is defined as
\[
\begin{split}
    L_\Delta u(x) & =\frac{d}{ds}\bigg|_{s=0} (-\Delta)^s u = c_N \int_{B_1(x)} \frac{u(x)-u(y)}{\abs{x-y}^N} ~dy -c_N\int_{\R^N\setminus\mathcal{B}_1(x)} \frac{u(y)}{\abs{x-y}^N} ~dy + \rho_Nu(x),
    \end{split}
\]
where $B_1(x) \subset \mathbb{R}^N$ denotes the unit ball centered at $x$ and $c_N, \rho_N$ are normalizing constants. The foundational analysis in \cite{Chen-Weth-CPDE-2019} established a variational setting for Dirichlet problems involving the logarithmic Laplacian $L_{\Delta}$. Exploiting the underlying Hilbert space structure. The authors obtained a spectral characterization of the eigenvalues of $L_\Delta$ and showed that the spectrum $\sigma(L_\Delta)$ of $L_\Delta$ consists of eigenvalues $(\lambda_k)$ satisfying
$$\lambda_1 \leq \lambda_2 \leq \lambda_3 \leq \cdots \leq \lambda_k \leq \cdots \quad \text{and} \quad \lambda_k \to \infty \ \text{as} \ k \to \infty.$$

Subsequent developments in this direction include the study of logarithmic Sobolev inequalities, as well as optimal continuous and compact embeddings of zero-order Sobolev spaces, carried out by Salda\~na et al. \cite{Angeles-Saldana, Santamaria-Saldana} and by Arora et al. \cite{Arora-Giacomoni-Vaishnavi}, respectively. In \cite{Angeles-Saldana, Arora-Giacomoni-Vaishnavi, Santamaria-Saldana}, the limit behavior of equations
involving the fractional Laplacian $(-\Delta)^{s}$ as $s \to 0$ has also been studied, with connections to models in population dynamics, optimal control, and image processing.  Additional advancements include: regularity results \cite{Santamaria-Rios-Saldana, Lara-Saldana}, spectral asymptotics \cite{Chen-Veron, Feulefack-Jarohs-Weth}, and geometric interpretations related to $0$-fractional perimeters \cite{Luca et. al}.

Recently, Dyda et al. \cite{Dyda-Jarohs-Sk} introduced the nonlinear extension $L_{\Delta_p}$ of $L_{\Delta}$, which arises naturally as the first-order derivative of the fractional $p$-Laplacian. For $u \in C_c^\alpha(\mathbb{R}^N)$ with $\alpha>0$ and $x \in \mathbb{R}^N$, the logarithmic $p$-Laplacian $L_{\Delta_p}$ is defined by
\[
\begin{aligned}
L_{\Delta_p} u(x)
&= C_{N,p}\int_{B_1(x)}
    \frac{|u(x)-u(y)|^{p-2}(u(x)-u(y))}{|x-y|^{N}}\, dy \\
&\quad + C_{N,p}\int_{\mathbb{R}^N\setminus B_1(x)}
    \frac{|u(x)-u(y)|^{p-2}(u(x)-u(y))
          - |u(x)|^{p-2}u(x)}{|x-y|^N}\, dy \\
&\quad + \rho_N\, |u(x)|^{p-2}u(x), \quad x \in \mathbb{R}^N
\end{aligned}
\]
where $C_{N,p}$ and $\rho_N$ are normalizing constants (see \cite[Theorem 1.1]{Dyda-Jarohs-Sk}). The appropriate function space to look for solutions to the Dirichlet problem involving $L_{\Delta_p}$ was recently introduced in Foghem \cite{Foghem-Nodea-2025}. Set
\[
[u]_p^p = C_{N,\,p} \iint_{\substack{\R^N \times \R^N\\[1pt] |x - y| < 1}} \frac{|u(x) - u(y)|^p}{|x - y|^N}\, dxdy
\]
and let
\[
X^p_0(\Omega) = \set{u \in L^p(\R^N) : u = 0 \text{ in } \R^N \setminus \Omega \text{ and } [u]_p < \infty}.
\]
Then $X^p_0(\Omega)$ endowed with the norm
\[
\norm[X^p_0(\Omega)]{u} = \left(\pnorm{u}^p + [u]_p^p\right)^{1/p}
\]
is a uniformly convex Banach space, where $\pnorm{\cdot}$ denotes the norm in $L^p(\R^N)$, and
\[
\norm{u} = [u]_p
\]
defines an equivalent norm on $X^p_0(\Omega)$. Moreover, $X^p_0(\Omega)$ is compactly embedded in $L^p(\Omega)$. By exploiting the compact embeddings, the authors in \cite{Dyda-Jarohs-Sk} proved that the first eigenvalue $\lambda_1$ of $L_{\Delta_p}$ is simple, and that its corresponding eigenfunction is bounded and strictly positive in $\Omega$. However, the structure and characterization of the higher eigenvalues remain open.

The main difficulties stem from the nonlinear nature of the operator $L_{\Delta_p}$, which is analogous to those encountered in the study of classical nonlinear operators, and the sign-changing nature of the associated bilinear form. For the $p$-Laplacian and the fractional $p$-Laplacian, Perera et al. \cite{Iannizzotto et al-ACV-2016, Perera-TMNA-2003} introduced a new variational construction of eigenvalues based on the $\mathbb{Z}_2$-cohomological index of Fadell and Rabinowitz. This approach differs from the traditional framework, which relies on the Krasnosel’skii genus as the latter lacks certain Morse-theoretic features. The refined structure of this new sequence of eigenvalues has been instrumental in extending many standard variational techniques for semilinear problems to the quasilinear setting (see \cite{Duzgun-Iannizzotto, Frassu-Iannizzotto, Perera-TMNA-2003,  Perera-Szulkin-2005, MR2640827}).

Motivated by these developments, and employing a minimax scheme built upon the $\Z_2$-cohomological index of Fadell and Rabinowitz, we establish an unbounded sequence of minimax eigenvalues $(\lambda_k)$ for our operator $L_{\Delta_p}$, thereby addressing the aforementioned gap in the first part of the present work.

Very recently, Arora et al. \cite{AroGiaHajVai-2025} proved a new $p$-logarithmic Sobolev inequality and derived optimal continuous embeddings of $X^p_0(\Omega)$ in Orlicz-type spaces $L^\varphi(\Omega)$ given by the modular function $\varphi(t) \approx t^p \ln(t)$ for large $t \gg 1$ and compact embeddings in Orlicz-type spaces $L^\psi(\Omega)$ given by the modular function $\psi(t) $ such that $\lim_{t \to \infty}\frac{\psi(t)}{\varphi(t)} =0$. As an application of these embedding results, they first proved existence results for a class of Dirichlet boundary value problems involving the logarithmic $p$-Laplacian and critical growth nonlinearities.

Next, as an application of these characterizations and $p$-logarithmic Sobolev inequality proved in Arora et al. \cite{AroGiaHajVai-2025}, we study a Dirichlet boundary value problem involving the logarithmic $p$-Laplacian and a class of nonlinearities exhibiting $p$-superlinear and subcritical growth at infinity, addressing a problem that has remained open in the literature even for $p=2$.

To the best of our knowledge, in the nonlinear case $p \neq 2$, \cite{AroGiaHajVai-2025} and \cite{Dyda-Jarohs-Sk} are the only studies to date devoted to the
logarithmic $p$-Laplacian. Research on the logarithmic $p$-Laplacian is still in its early stages of development. Many questions concerning its analytical properties remain largely open, and the present work aims to advance this emerging theory.

\section{Main problem}
Motivated by the above discussion, in this work, we aim to study the following class of Dirichlet boundary value problems
\begin{equation} \label{main:problem}
\left\{\begin{aligned}
L_{\Delta_p} u & = g(u)\, && \text{in } \Omega\\[7.5pt]
u & = 0 && \text{in } \R^N \setminus \Omega,
\end{aligned}\right.
\end{equation}
where $\Omega$ is a bounded domain in $\mathbb{R}^N$ and $N \geq 1$. The continuous function $g: \mathbb{R} \to \mathbb{R}$ divide our study into two categories depending upon its growth:
\begin{itemize}
    \item when $g(u) = \lambda |u|^{p-2} u$. Then, the problem \eqref{main:problem} can be regarded as an eigenvalue problem for the logarithmic $p$-Laplacian.

    In this case, we show the existence of an unbounded sequence of minimax eigenvalues $(\lambda_k)$ via the $\Z_2$-cohomological index of Fadell and Rabinowitz and also demonstrate the existence of a non-trivial critical group. For a detailed statement of results, we refer to Theorem \ref{Theorem 1} and Theorem \ref{Theorem 2}.
    \item when $g$ satisfies
    \begin{properties}{g}
    \item \label{g1} $\lim\limits_{t \to 0}\, \dfrac{g(t)}{|t|^{p-2}\, t} = \lambda \in \R$,
    \item \label{g2} $\lim\limits_{|t| \to \infty}\, \dfrac{g(t)}{|t|^{p-2}\, t \ln |t|} = 0$.
    \item \label{g3} there exist constants $\beta \in (0,1)$, and $t_0 > 1$ such that
    \[
    0 < \frac{\beta|t|^p}{\ln |t|} \le tg(t) - p G(t) \quad \text{for } |t| \ge t_0 \quad \text{where} \ G(t) = \int_0^t g(s) ~ds.
    \]
\end{properties}
The assumptions \ref{g1} and \ref{g2} on the function $g$ are motivated by the $p$-logarithmic Sobolev inequality and optimal embeddings results proved in Arora et al. \cite{AroGiaHajVai-2025} signifying the $p$-superlinear and subcritical growth at infinity. The assumption \ref{g3} introduced above can be interpreted as the logarithmic counterpart to the classical Ambrosetti-Rabinowitz condition famously used in the literature to treat the quasilinear Dirichlet boundary value problems involving classical $p$-Laplacian and fractional $p$-Laplacian (see \cite{Ambrosetti-Rabinowitz, Duzgun-Iannizzotto, Frassu-Iannizzotto}). For example, the continuous functions $h_1$, $h_{2}$ and $h_{3}$ given by
\[
h_1(t) = \lambda |t|^{p-2} t (\ln(e+|t|))^\theta, \quad h_{2}(t) = \lambda |t|^{p-2} t + |t|^{p-2} t (\ln(1+|t|))^\theta
\]
and
\[
h_{3}(t) =
\begin{cases}
    \lambda |t|^{p-2} t \ & \text{if} \quad |t| \leq t_1,\\
    |t|^{p-2}\, t \abs{\ln |t|}^\theta \ & \text{if} \quad |t| \geq t_0,
\end{cases}
\]
satisfies \ref{g1}-\ref{g3}, where $\lambda \in \mathbb{R}$, $0 < t_1 < t_0$ and $\theta \in (0,1).$ These examples show that problem \eqref{main:problem} naturally includes eigenvalue problems for the operator $L_{\Delta_p}$ with logarithmic superlinear nonlinearities, as well as perturbed eigenvalue problems featuring subcritical logarithmic growth, in light of the embedding results established in Arora et al. \cite{AroGiaHajVai-2025}.

By employing a variant of the linking theorem due to Yang et al. \cite{MR3616328} together with the consequences of $p$-logarithmic Sobolev inequality established by Arora et al. \cite{AroGiaHajVai-2025}, we prove the existence of a nontrivial solution of the problem \eqref{main:problem} in both the cases $\lambda < \lambda_1$ and $\lambda \in (\lambda_k, \lambda_{k+1})$, $k \in \mathbb{N}$, where the parameter $\lambda$ is given by the assumption \ref{g1}. For a detailed statement of results, we refer to Theorem \ref{thm:existence:lambda1} and Theorem \ref{thm:existence:lambdak}.
\end{itemize}

\textbf{Outline of the paper}: The rest of the paper is organized as follows. In Section \ref{sec:variational setting}, we show that the potential operator $A_p$ (see \eqref{poten-opera}) associated with the operator $L_{\Delta_p}$ is of type $(S)$. This property plays a crucial role in establishing the compactness condition for the energy functional corresponding to the problem \eqref{main:problem}. In Section \ref{sec:eigen}, we construct an unbounded sequence of minimax eigenvalues $(\lambda_k)$ and prove the existence of a nontrivial critical group (see Theorems \ref{Theorem 1} and \ref{Theorem 2}). In Section \ref{sec:new-log-ineq}, we derive new logarithmic estimates that are essential for verifying the Cerami compactness condition and for obtaining a precise understanding of the behaviour of the energy functional in a neighbourhood of the origin. Section \ref{sec:nonlinear-problem} is devoted to the study of problem \eqref{main:problem}, where we establish the existence of a nontrivial solution depending on the parameter $\lambda$. The case $\lambda < \lambda_1$ is addressed in Theorem \ref{thm:existence:lambda1}, while the case $\lambda \in (\lambda_k, \lambda_{k+1})$, $k\in\mathbb{N}$, is treated in Theorem \ref{thm:existence:lambdak}.

\section{Operator setting}\label{sec:variational setting}
For $u, v \in X^p_0(\Omega)$, let
\begin{multline*}
\E_{L,\,p}(u,v) = \frac{C_{N,\,p}}{2} \int_{\R^N} \int_{B_1(x)} \frac{|u(x) - u(y)|^{p-2}\, (u(x) - u(y))(v(x) - v(y))}{|x - y|^N}\, dydx\\[7.5pt]
+ \frac{C_{N,\,p}}{2} \int_{\R^N} \int_{\R^N \setminus B_1(x)} \frac{1}{|x - y|^N}\, \big[|u(x) - u(y)|^{p-2}\, (u(x) - u(y))(v(x) - v(y))\\[7.5pt]
- |u(x)|^{p-2}\, u(x)v(x) - |u(y)|^{p-2}\, u(y)v(y)\big]\, dydx + \rho_N(p) \int_{\R^N} |u(x)|^{p-2}\, u(x) v(x)\, dx.
\end{multline*}
We consider the nonlinear operator equation
\[
\Ap[u] = l_f
\]
in the dual $X^p_0(\Omega)^\ast$ of $X^p_0(\Omega)$, where $\Ap \in C(X^p_0(\Omega),X^p_0(\Omega)^\ast)$ is the operator given by
\begin{equation}\label{poten-opera}
    \dualp{\Ap[u]}{v} = \E_{L,\,p}(u,v), \quad u, v \in X^p_0(\Omega)
\end{equation}
and $l_f \in X^p_0(\Omega)^\ast$ is given by
\[
l_f(v) = \int_\Omega fv\, dx, \quad v \in X^p_0(\Omega), \quad \text{where} \ f \in L^{p/(p-1)}(\Omega).
\]
The operator $\Ap$ is a potential operator with the potential
\[
I_p(u) = \frac{1}{p} \dualp{\Ap[u]}{u} = \frac{1}{p}\, \E_{L,\,p}(u,u), \quad u \in X^p_0(\Omega),
\]
i.e., $I_p'(u) = \Ap[u]$ for all $u \in X^p_0(\Omega)$. We write $\Ap = \Ap' + \Ap''$, where
\[
\dualp{\Ap'u}{v} = \frac{C_{N,\,p}}{2} \iint_{\substack{\R^N \times \R^N\\[1pt] |x - y| < 1}} \frac{|u(x) - u(y)|^{p-2}\, (u(x) - u(y))(v(x) - v(y))}{|x - y|^N}\, dydx
\]
and
\begin{multline*}
\dualp{\Ap''u}{v} = \frac{C_{N,\,p}}{2} \iint_{\substack{\R^N \times \R^N\\[1pt] |x - y| \ge 1}} \frac{1}{|x - y|^N}\, \big[|u(x) - u(y)|^{p-2}\, (u(x) - u(y))(v(x) - v(y))\\[7.5pt]
\hspace{-5.5pt} - |u(x)|^{p-2}\, u(x)v(x) - |u(y)|^{p-2}\, u(y)v(y)\big]\, dydx + \rho_N(p) \int_{\R^N} |u(x)|^{p-2}\, u(x) v(x)\, dx
\end{multline*}
for $u, v \in X^p_0(\Omega)$. Note that, we have
\begin{equation} \label{15}
\dualp{\Ap'u}{u} = \half \norm{u}^p
\end{equation}
and
\begin{equation} \label{16}
\dualp{\Ap'u}{v} \le \half \norm{u}^{p-1} \norm{v} \quad \forall u, v \in X^p_0(\Omega)
\end{equation}
by the H\"{o}lder inequality. Next, we recall the following estimate obtained in Arora et al.\! \cite{AroGiaHajVai-2025} and prove some preliminaries results concerning the operator $A_p$.
\begin{lemma}[{\cite[Lemma 4.1]{AroGiaHajVai-2025}}] \label{Lemma 2}
There exists a constant $C > 0$ such that
\[
\big|\!\dualp{\Ap''u}{v}\!\big| \le C \pnorm{u}^{p-1} \pnorm{v} \quad \forall u, v \in X^p_0(\Omega).
\]
\end{lemma}
\begin{lemma} \label{Lemma 1}
There exists a constant $C > 0$ such that
\[
\abs{\E_{L,\,p}(u,u) - \half \norm{u}^p} \le C \pnorm{u}^p \quad \forall u \in X^p_0(\Omega).
\]
\end{lemma}
\begin{proof}
  Since
\[
\E_{L,\,p}(u,u) = \dualp{\Ap[u]}{u} = \dualp{\Ap'u}{u} + \dualp{\Ap''u}{u}\!,
\]
Lemma \ref{Lemma 2} together with \eqref{15} gives the required estimate for $\E_{L,\,p}$.
\end{proof}
\begin{lemma} \label{Lemma 3}
The operator $\Ap$ is of type (S), i.e., every sequence $\seq{u_j} \subset X^p_0(\Omega)$ such that $u_j \wto u$ and $\dualp{\Ap[u_j]}{u_j - u} \to 0$ converges strongly to $u$.
\end{lemma}
\begin{proof}
Since $X^p_0(\Omega)$ is uniformly convex, it follows from \eqref{15} and \eqref{16} that the operator $\Ap'$ is of type (S) (see Perera et al.\! \cite[Proposition 1.3]{MR2640827}). Moreover, since $u_j \wto u$ and $X^p_0(\Omega)$ is compactly embedded in $L^p(\Omega)$, $u_j \to u$ in $L^p(\Omega)$. Then $\dualp{\Ap''u_j}{u_j - u} \to 0$ by Lemma \ref{Lemma 2} and hence
\[
\dualp{\Ap'u_j}{u_j - u} = \dualp{\Ap[u_j]}{u_j - u} - \dualp{\Ap''u_j}{u_j - u} \to 0.
\]
The desired conclusion follows since $\Ap'$ is of type (S).
\end{proof}

\section{Eigenvalue problem}\label{sec:eigen}
Let $\Omega$ be a bounded domain in $\mathbb{R}^N$. Consider the Dirichlet eigenvalue problem
\begin{equation} \label{2}\tag{$L_{p}$}
\left\{\begin{aligned}
L_{\Delta_p} u & = \lambda\, |u|^{p-2}\, u && \text{in } \Omega\\[7.5pt]
u & = 0 && \text{in } \R^N \setminus \Omega.
\end{aligned}\right.
\end{equation}
The above eigenvalue problem can be written as the operator equation
\[
\Ap[u] = \lambda \Bp[u]
\]
in $X^p_0(\Omega)^\ast$, the dual of $X^p_0(\Omega)$, where $\Bp \in C(X^p_0(\Omega),X^p_0(\Omega)^\ast)$ is given by
\[
\dualp{\Bp[u]}{v} = \int_\Omega |u|^{p-2}\, uv\, dx, \quad u, v \in X^p_0(\Omega).
\]
The operator $\Bp$ is a potential operator with the potential
\[
J_p(u) = \frac{1}{p} \dualp{\Bp[u]}{u} = \frac{1}{p} \int_\Omega |u|^p \, dx, \quad u \in X^p_0(\Omega),
\]
i.e., $J_p' = \Bp$. Note that $\Bp$ is a compact operator since $X^p_0(\Omega)$ is compactly embedded in $L^p(\Omega)$. Define, the restricted potential operator $\I_{p}: \M_p \to \mathbb{R}$ as
\begin{equation} \label{26}
\I_{p} = \restr{I_p}{\M_p}  \quad \text{where} \quad \M_p = \set{u \in X^p_0(\Omega) : J_p(u) = 1} = \{u \in X^p_0(\Omega) : \pnorm{u}^p = p\}.
\end{equation}
Then the norm of $\I_{p}'(u)$ as an element of the cotangent space to $\M_p$ at $u$ is given by
\[
\dnorm[u]{\I_{p}'(u)} = \min_{\mu \in \R}\, \dnorm{I_p'(u) - \mu J_p'(u)} = \min_{\mu \in \R}\, \dnorm{\Ap[u] - \mu \Bp[u]},
\]
where $\dnorm{\cdot}$ denotes the norm in $X^p_0(\Omega)^\ast$ (see, e.g., Perera et al.\! \cite[Proposition 3.54]{MR2640827}). Therefore, $\I_{p}'(u) = 0$ if and only if $\Ap[u] = \mu \Bp[u]$ for some Lagrange multiplier $\mu \in \R$, in which case
\[
\mu = \frac{\dualp{\Ap[u]}{u}}{\dualp{\Bp[u]}{u}} = \frac{I_p(u)}{J_p(u)} = \I_{p}(u).
\]
Hence, we have the following variational formulation of the eigenvalue problem \eqref{2}.

\begin{lemma} \label{Lemma 4}
The eigenvalues of \eqref{2} coincide with the critical values of $\I_{p}$, i.e., $\lambda$ is an eigenvalue if and only if there is a $u \in \M$ such that $\I_{p}'(u) = 0$ and $\I_{p}(u) = \lambda$.
\end{lemma}

It was shown in \cite{Dyda-Jarohs-Sk} that the first eigenvalue
\begin{equation} \label{3}
\lambda_1 = \inf_{u \in \M_p}\, \I_{p}(u)
\end{equation}
is simple and has a corresponding eigenfunction that is bounded and strictly positive in $\Omega$. Next, we will construct an increasing and unbounded sequence $\seq{\lambda_k}$ of eigenvalues using a suitable minimax scheme. First, we show that the functional $\I_{p}$ satisfies the Palais-Smale compactness condition.

\begin{lemma} \label{Lemma 5}
For all $c \in \R$, $\I_{p}$ satisfies the {\em \PS{c}} condition, i.e., every sequence $\seq{u_j} \subset \M_p$ such that $\I_{p}(u_j) \to c$ and $\I_{p}'(u_j) \to 0$ has a strongly convergent subsequence.
\end{lemma}

\begin{proof}
Let $\seq{u_j} \subset \M_p$ such that $\I_{p}(u_j) \to c$ for some $c \in \mathbb{R}$ and $\I_{p}'(u_j) \to 0.$ Then, we have
\[
\I_{p}(u_j) = \frac{1}{p} \dualp{\Ap[u_j]}{u_j} = \frac{1}{p} \dualp{\Ap'u_j}{u_j} + \frac{1}{p} \dualp{\Ap''u_j}{u_j} = \half \norm{u_j}^p + \frac{1}{p} \dualp{\Ap''u_j}{u_j}  \to c.
\]
By Lemma \ref{Lemma 2}, we have
\[
\big|\!\dualp{\Ap''u_j}{u_j}\!\big| \le C \pnorm{u_j}^p = Cp.
\]
This implies that $\seq{u_j}$ is bounded in $X^p_0(\Omega)$. Since $X^p_0(\Omega)$ is reflexive and $X^p_0(\Omega)$ is compactly embedded in $L^p(\Omega)$, then a renamed subsequence of $\seq{u_j}$ converges weakly to some $u \in X^p_0(\Omega)$ and $u_j \to u$ in $L^p(\Omega)$. Since
\[
\dnorm[u_j]{\I_{p}'(u_j)} = \min_{\mu \in \R}\, \dnorm{\Ap[u_j] - \mu \Bp[u_j]} \to 0,
\]
we have
\begin{equation}\label{00}
    \dualp{\Ap[u_j]}{v} - \mu_j \dualp{\Bp[u_j]}{v} = \o(\norm{v}) \quad \forall v \in X^p_0(\Omega)
\end{equation}
for some sequence of Lagrange multipliers $\seq{\mu_j} \subset \R$. Since $\seq{u_j}$ is bounded in $X^p_0(\Omega)$ and
\[
\dualp{\Ap[u_j]}{u_j} = p\, \I(u_j) \to pc, \qquad \dualp{\Bp[u_j]}{u_j} = \pnorm{u_j}^p = p,
\]
taking $v = u_j$ shows that $\mu_j \to c$. Now taking $v = u_j - u$ and by applying H\"older inequality, we obtain
\[
\abs{\dualp{\Bp[u_j]}{u_j - u}} = \abs{\int_\Omega |u_j|^{p-2}\, u_j\, (u_j - u)\, dx} \le \pnorm{u_j}^{p-1} \pnorm{u_j - u} \to 0.
\]
This further in view of \eqref{00} implies that $\dualp{\Ap[u_j]}{u_j - u} \to 0$. Finally, by Lemma \ref{Lemma 3}, we get the desired conclusion.
\end{proof}

Next, we construct an unbounded sequence of minimax eigenvalues of problem \eqref{2}. Although this can be done using Krasnoselskii's genus, we prefer to use the $\Z_2$-cohomological index of Fadell and Rabinowitz (see \cite{MR0478189}) in order to obtain nontrivial critical groups and construct linking sets.

\begin{definition}[Fadell and Rabinowitz \cite{MR0478189}] \label{Definition 1}
Let $W$ be a Banach space and let $\A$ denote the class of symmetric subsets of $W \setminus \set{0}$. For $A \in \A$, let $\overline{A} = A/\Z_2$ be the quotient space of $A$ with each $u$ and $-u$ identified, let $f : \overline{A} \to \RP^\infty$ be the classifying map of $\overline{A}$, and let $f^\ast : H^\ast(\RP^\infty) \to H^\ast(\overline{A})$ be the induced homomorphism of the Alexander-Spanier cohomology rings. The cohomological index of $A$ is defined by
\[
i(A) = \begin{cases}
0 & \text{if } A = \emptyset\\[5pt]
\sup \set{m \ge 1 : f^\ast(\omega^{m-1}) \ne 0} & \text{if } A \ne \emptyset,
\end{cases}
\]
where $\omega \in H^1(\RP^\infty)$ is the generator of the polynomial ring $H^\ast(\RP^\infty) = \Z_2[\omega]$.
\end{definition}

\begin{example}
The classifying map of the unit sphere $S^N$ in $\R^{N+1},\, N \ge 0$ is the inclusion $\RP^N \incl \RP^\infty$, which induces isomorphisms on the cohomology groups $H^l$ for $l \le N$, so $i(S^N) = N + 1$.
\end{example}
The following proposition summarizes the basic properties of the cohomological index.
\begin{proposition}[Fadell and Rabinowitz \cite{MR0478189}] \label{Proposition 7}
The index $i : \A \to \N \cup \set{0,\infty}$ has the following properties:
\begin{enumerate}
\item[$(i_1)$] Definiteness: $i(A) = 0$ if and only if $A = \emptyset$.
\item[$(i_2)$] Monotonicity: If there is an odd continuous map from $A$ to $B$ (in particular, if $A \subset B$), then $i(A) \le i(B)$. Thus, equality holds when the map is an odd homeomorphism.
\item[$(i_3)$] Dimension: $i(A) \le \dim W$.
\item[$(i_4)$] Continuity: If $A$ is closed, then there is a closed neighborhood $N \in \A$ of $A$ such that $i(N) = i(A)$. When $A$ is compact, $N$ may be chosen to be a $\delta$-neighborhood $N_\delta(A) = \set{u \in W : \dist(u,A) \le \delta}$.
\item[$(i_5)$] Subadditivity: If $A$ and $B$ are closed, then $i(A \cup B) \le i(A) + i(B)$.
\item[$(i_6)$] Stability: If $\Sigma A$ is the suspension of $A \ne \emptyset$, obtained as the quotient space of $A \times [-1,1]$ with $A \times \set{1}$ and $A \times \set{-1}$ collapsed to different points, then $i(\Sigma A) = i(A) + 1$.
\item[$(i_7)$] Piercing property: If $C$, $C_0$, and $C_1$ are closed and $\varphi : C \times [0,1] \to C_0 \cup C_1$ is a continuous map such that $\varphi(-u,t) = - \varphi(u,t)$ for all $(u,t) \in C \times [0,1]$, $\varphi(C \times [0,1])$ is closed, $\varphi(C \times \set{0}) \subset C_0$, and $\varphi(C \times \set{1}) \subset C_1$, then $i(\varphi(C \times [0,1]) \cap C_0 \cap C_1) \ge i(C)$.
\item[$(i_8)$] Neighborhood of zero: If $U$ is a bounded closed symmetric neighborhood of $0$, then $i(\bdry{U}) = \dim W$.
\end{enumerate}
\end{proposition}

In view of Lemma \ref{Lemma 5}, we have the following theorem (see Perera et al.\! \cite[Proposition 3.52, Proposition 3.53 and Theorem 4.6]{MR2640827}).

\begin{theorem} \label{Theorem 1}
Let $\F$ denote the class of symmetric subsets of $\M_p$. For $k \ge 1$, let
\[
\F_k = \set{M \in \F : i(M) \ge k}
\]
and set
\[
\lambda_k := \inf_{M \in \F_k}\, \sup_{u \in M}\, \I_{p}(u).
\]
Then $\lambda_k \nearrow \infty$ is a sequence of eigenvalues of problem \eqref{2}.
\begin{enumroman}
\item \label{Theorem 1.i} If $\lambda_k = \dotsb = \lambda_{k+m-1} = \lambda$ and $E_\lambda$ is the set of eigenfunctions associated with $\lambda$ that lie on $\M_p$, then
    \[
    i(E_\lambda) \ge m.
    \]
\item \label{Theorem 1.iii} If $\lambda_k < \lambda < \lambda_{k+1}$, then
    \[
    i(\I^{\lambda_k}) = i(\M \setminus \I_{\lambda}) = i(\I^\lambda) = i(\M \setminus \I_{\lambda_{k+1}}) = k,
    \]
    where $\I^a = \set{u \in \M : \I_{p}(u) \le a}$ and $\I_{a} = \set{u \in \M : \I_{p}(u) \ge a}$ for $a \in \R$.
\end{enumroman}
\end{theorem}

\begin{remark}
Problem \eqref{2} may have eigenvalues other than those of the sequence $\seq{\lambda_k}$.
\end{remark}

The variational functional associated with the eigenvalue problem \eqref{2} is
\[
\Phi_\lambda(u) = I_p(u) - \lambda J_p(u), \quad u \in X^p_0(\Omega)
\]
(see the proof of \cite[Lemma 4.3]{AroGiaHajVai-2025}). When $\lambda$ is not an eigenvalue, the origin is the only critical point of $\Phi_\lambda$ and its critical groups there are given by
\[
C^l(\Phi_\lambda,0) = H^l(\Phi_\lambda^0,\Phi_\lambda^0 \setminus \set{0}), \quad l \ge 0,
\]
where $\Phi_\lambda^0 = \set{u \in X^p_0(\Omega) : \Phi_\lambda(u) \le 0}$. We have the following theorem, where $\widetilde{H}^\ast$ denotes reduced cohomology.

\begin{theorem} \label{Theorem 2}
Assume that $\lambda$ is not an eigenvalue.
\begin{enumroman}
\item \label{Theorem 2.i} If $\lambda < \lambda_1$, then $C^l(\Phi_\lambda,0) \isom \delta_{l0}\, \Z_2$, where $\delta$ denotes the Kronecker delta.
\item \label{Theorem 2.ii} If $\lambda > \lambda_1$, then $C^l(\Phi_\lambda,0) \isom \widetilde{H}^{l-1}(\I^\lambda)$, in particular, $C^0(\Phi_\lambda,0) = 0$.
\item \label{Theorem 2.iii} If $\lambda_k < \lambda < \lambda_{k+1}$, then $C^k(\Phi_\lambda,0) \ne 0$.
\end{enumroman}
In particular, $C^l(\Phi_\lambda,0)$ is nontrivial for some $l$.
\end{theorem}

\begin{proof}
If $u \in \Phi_\lambda^0$, then $tu \in \Phi_\lambda^0$ for all $t \in [0,1]$ since $\Phi_\lambda(tu) = t^p \Phi_\lambda(u)$. This implies that $\Phi_\lambda^0$ contracts to $\set{0}$ via
\[
\Phi_\lambda^0 \times [0,1] \to \Phi_\lambda^0, \quad (u,t) \mapsto (1 - t)u
\]
and $\Phi_\lambda^0 \setminus \set{0}$ deformation retracts to $\Phi_\lambda^0 \cap \M_p$ via
\[
(\Phi_\lambda^0 \setminus \set{0}) \times [0,1] \to \Phi_\lambda^0 \setminus \set{0}, \quad (u,t) \mapsto (1 - t)u + \frac{p^{1/p}\, tu}{\pnorm{u}}.
\]
It follows that
\[
C^l(\Phi_\lambda,0) \isom \begin{cases}
\delta_{l0}\, \Z_2 & \text{if } \Phi_\lambda^0 \cap \M_p = \emptyset\\[7.5pt]
\widetilde{H}^{l-1}(\Phi_\lambda^0 \cap \M_p) & \text{if } \Phi_\lambda^0 \cap \M_p \ne \emptyset
\end{cases}
\]
(see, e.g., Perera et al.\! \cite[Proposition 2.4]{MR2640827}). Since
\[
\Phi_\lambda(u) = \I_{p}(u) - \lambda
\]
for $u \in \M_p$, $\Phi_\lambda^0 \cap \M_p = \I^\lambda$, so we have
\[
C^l(\Phi_\lambda,0) \isom \begin{cases}
\delta_{l0}\, \Z_2 & \text{if } \I^\lambda = \emptyset\\[7.5pt]
\widetilde{H}^{l-1}(\I^\lambda) & \text{if } \I^\lambda \ne \emptyset.
\end{cases}
\]
Since $\I^\lambda = \emptyset$ if and only if $\lambda < \lambda_1$ by \eqref{3}, \ref{Theorem 2.i} and \ref{Theorem 2.ii} follow from this. If $\lambda_k < \lambda < \lambda_{k+1}$, then $i(\I^\lambda) = k$ by Theorem \ref{Theorem 1} \ref{Theorem 1.iii} and hence $\widetilde{H}^{k-1}(\I^\lambda) \ne 0$ by Perera et al.\! \cite[Proposition 2.14 ({\em iv})]{MR2640827}, so \ref{Theorem 2.iii} follows from \ref{Theorem 2.ii}.
\end{proof}


\section{$p$-logarithmic Sobolev inequality}\label{sec:new-log-ineq}
In this section, we derive a series of preliminary results that rely on the following $p$-logarithmic Sobolev inequality established in Arora et al.\! \cite{AroGiaHajVai-2025}.

\begin{theorem}[{\cite[Theorem 2.5]{AroGiaHajVai-2025}}]
For all $u \in X^p_0(\Omega)$,
\begin{equation} \label{4}
\frac{p^2}{N} \int_\Omega |u|^p \ln |u|\, dx \le \E_{L,\,p}(u,u) + \frac{p^2}{N} \pnorm{u}^p \ln \pnorm{u} + k_0(N,p) \pnorm{u}^p
\end{equation}
for a suitable constant $k_0(N,p)$.
\end{theorem}

We have the following corollaries.

\begin{corollary} \label{Corollary 1}
For any $\delta > 0$, there exists a constant $C_\delta > 0$ such that
\begin{equation} \label{20}
\int_\Omega |u|^p \abs{\ln |u|} dx \le C_\delta\, \big(\norm{u}^p + \pnorm{u}^{p + \delta} + 1\big) \quad \forall u \in X^p_0(\Omega).
\end{equation}
\end{corollary}

\begin{proof}
This estimate is immediate from the inequality \eqref{4} and Lemma \ref{Lemma 1}.
\end{proof}

\begin{corollary} \label{Corollary 2}
For all $\gamma \in (0,1)$,
\[
\int_\Omega |u|^p\, (\ln (1 + |u|))^\gamma\, dx = \o(\norm{u}^p) \quad \text{as } \norm{u} \to 0.
\]
\end{corollary}

\begin{proof}
For any $\eps > 0$, there exists a constant $C_\eps > 1$ such that
\begin{equation} \label{23}
(\ln (1 + |t|))^\gamma < \eps \ln |t| \quad \text{for } |t| > C_\eps.
\end{equation}
For $u \in X^p_0(\Omega) \setminus \set{0}$ with $\rho := \norm{u} \le 1$, set $v = u/\rho$. Then
\begin{multline*}
\frac{1}{\norm{u}^p} \int_\Omega |u|^p\, (\ln (1 + |u|))^\gamma\, dx = \int_\Omega |v|^p\, (\ln (1 + \rho\, |v|))^\gamma\, dx \le \int_{\set{|v| \le C_\eps}} |v|^p\, (\ln (1 + \rho\, |v|))^\gamma\, dx\\[7.5pt]
+ \int_{\set{|v| > C_\eps}} |v|^p\, (\ln (1 + |v|))^\gamma\, dx \le C_\eps^p\, (\ln (1 + \rho\, C_\eps))^\gamma \vol{\Omega} + \eps \int_\Omega |v|^p \abs{\ln |v|} dx
\end{multline*}
by \eqref{23}, where $\vol{\Omega}$ is the volume of $\Omega$. Since $\norm{v} = 1$, the last integral is bounded by Corollary \ref{Corollary 1} and the embedding $X^p_0(\Omega) \hookrightarrow L^p(\Omega)$. Since $\eps > 0$ is arbitrary, the desired conclusion follows by taking the limit $\rho \to 0$.
\end{proof}

\begin{corollary} \label{Corollary 3}
There exists a constant $C > 0$ such that
\[
\int_{\set{|u| > 1}} |u|^p \ln |u|\, dx \le C \norm{u}^p
\]
for all $u \in X^p_0(\Omega)$ with $\pnorm{u} \le 1$.
\end{corollary}

\begin{proof}
The inequality clearly holds when $u = 0$, so suppose $0 < \pnorm{u} \le 1$. We have
\begin{equation} \label{24}
\int_{\set{|u| > 1}} |u|^p \ln |u|\, dx \le \frac{N}{p^2} \left[\E_{L,\,p}(u,u) + k_0(N,p) \pnorm{u}^p\right] + \pnorm{u}^p \ln \pnorm{u} - \int_{\set{|u| \le 1}} |u|^p \ln |u|\, dx
\end{equation}
by \eqref{4}. Since $\pnorm{u} \le 1$, $\ln \pnorm{u} \le 0$ and hence
\begin{multline}
\pnorm{u}^p \ln \pnorm{u} - \int_{\set{|u| \le 1}} |u|^p \ln |u|\, dx \le \left(\int_{\set{|u| \le 1}} |u|^p\, dx\right) \ln \pnorm{u} - \int_{\set{|u| \le 1}} |u|^p \ln |u|\, dx\\[7.5pt]
= - \int_{\set{|u| \le 1}} |u|^p \ln \frac{|u|}{\pnorm{u}}\, dx = \pnorm{u}^p \int_{\set{|u| \le 1}} (- v^p \ln v)\, dx,
\end{multline}
where $v = |u|/\pnorm{u}$. Since $- t^p \ln t \le 1/pe$ for all $t > 0$,
\begin{equation} \label{25}
\int_{\set{|u| \le 1}} (- v^p \ln v)\, dx \le \frac{\vol{\Omega}}{pe},
\end{equation}
where $\vol{\Omega}$ is the volume of $\Omega$. Combining \eqref{24}--\eqref{25} with Lemma \ref{Lemma 1} and the embedding $X^p_0(\Omega) \hookrightarrow L^p(\Omega)$ gives the desired inequality.
\end{proof}

\section{Problems with superlinear-subcritical nonlinearities}\label{sec:nonlinear-problem}
In this section, we consider the following nonlinear Dirichlet problem
\begin{equation}\label{5}
\left\{\begin{aligned}
L_{\Delta_p} u & = g(u) && \text{in } \Omega\\[7.5pt]
u & = 0 && \text{in } \R^N \setminus \Omega,
\end{aligned}\right.
\end{equation}
where $\Omega$ be a bounded domain in $\R^N$ and $g : \R \to \R$ is a continuous function satisfying  \ref{g1}-\ref{g3}.
It follows from \ref{g1} and \ref{g2} that for any $\eps > 0$, there exists a constant $C_\eps > 0$ such that
\begin{equation} \label{6}
|g(t)| \le \eps\, |t|^{p-1} \abs{\ln |t|} + C_\eps\, |t|^{p-1} \quad \forall t \in \R
\end{equation}
and
\[
|G(t)| \le \eps\, |t|^p \abs{\ln |t|} + C_\eps\, |t|^p \quad \forall t \in \R,
\]
where $G(t) = \int_0^t g(s)\, ds$.

\begin{definition}
We say that $u \in X^p_0(\Omega)$ is a weak solution of problem \eqref{5} if
\[
\E_{L,\,p}(u,v) = \int_\Omega g(u)\, v\, dx \quad \forall v \in X^p_0(\Omega).
\]
\end{definition}

The variational functional associated with the nonlinear problem \eqref{5} is
\begin{equation} \label{32}
\Phi(u) = \frac{1}{p}\, \E_{L,\,p}(u,u) - \int_\Omega G(u)\, dx, \quad u \in X^p_0(\Omega)
\end{equation}
(see the proof of \cite[Lemma 4.3]{AroGiaHajVai-2025}). Recall that $\Phi$ satisfies the Cerami compactness condition at the level $c \in \R$, or the \C{c} condition for short, if every sequence $\seq{u_j} \subset X^p_0(\Omega)$ such that $\Phi(u_j) \to c$ and $(1 + \norm{u_j})\, \Phi'(u_j) \to 0$, called a \C{c} sequence, has a strongly convergent subsequence. It suffices to show that $\seq{u_j}$ is bounded when verifying this condition by the following lemma.

\begin{lemma} \label{Lemma 6}
If \ref{g1} and \ref{g2} hold, then every bounded sequence $\seq{u_j} \subset X^p_0(\Omega)$ such that $\Phi'(u_j) \to 0$ has a strongly convergent subsequence.
\end{lemma}

\begin{proof}
We have
\begin{equation} \label{8}
\dualp{\Phi'(u_j)}{v} = \dualp{\Ap[u_j]}{v} - \int_\Omega g(u_j)\, v\, dx = \o(\norm{v}) \quad \forall v \in X^p_0(\Omega).
\end{equation}
Since $\seq{u_j}$ is bounded, a renamed subsequence converges to some $u$ weakly in $X^p_0(\Omega)$ and strongly in $L^p(\Omega)$. Taking $v = u_j - u$ in \eqref{8} gives
\[
\dualp{\Ap[u_j]}{u_j - u} = \int_\Omega g(u_j)\, (u_j - u)\, dx + \o(1).
\]
In view of Lemma \ref{Lemma 3}, now it suffices to show that the integral on the right-hand side goes to zero.

Let $\eps > 0$ and let $C_\eps > 0$ be as in \eqref{6}. Then
\begin{equation} \label{9}
\abs{\int_\Omega g(u_j)\, (u_j - u)\, dx} \le \eps \int_\Omega |u_j|^{p-1}\, |u_j - u| \abs{\ln |u_j|} dx + C_\eps \int_\Omega |u_j|^{p-1}\, |u_j - u|\, dx.
\end{equation}
We have
\[
\int_\Omega |u_j|^{p-1}\, |u_j - u| \abs{\ln |u_j|} dx \le \int_\Omega |u_j|^p \abs{\ln |u_j|} dx + \int_\Omega |u_j|^{p-1}\, |u| \abs{\ln |u_j|} dx.
\]
Since $|t|^{p-1} \abs{\ln |t|} \le 1/e\, (p - 1)$ for $|t| \le 1$,
\begin{multline*}
\int_\Omega |u_j|^{p-1}\, |u| \abs{\ln |u_j|} dx \le \frac{1}{e\, (p - 1)} \int_{\set{|u_j| \le 1}} |u|\, dx + \int_{\set{1 < |u_j| \le |u|}} |u|^p \ln |u|\, dx\\[7.5pt]
+ \int_{\set{|u_j| > |u|}} |u_j|^p \abs{\ln |u_j|} dx \le \frac{1}{e\, (p - 1)} \int_\Omega |u|\, dx + \int_\Omega |u|^p \abs{\ln |u|} dx + \int_\Omega |u_j|^p \abs{\ln |u_j|} dx.
\end{multline*}
Since $\seq{u_j}$ is bounded in $X^p_0(\Omega)$, $\int_\Omega |u_j|^p \abs{\ln |u_j|} dx$ is bounded by Corollary \ref{Corollary 1}, so it follows that $\int_\Omega |u_j|^{p-1}\, |u_j - u| \abs{\ln |u_j|} dx$ is bounded. On the other hand, $\int_\Omega |u_j|^{p-1}\, |u_j - u|\, dx \to 0$ as in the proof of Lemma \ref{Lemma 5}. So first letting $j \to \infty$ and then letting $\eps \to 0$ in \eqref{9} gives the desired conclusion.
\end{proof}


Next, in order to ensure that \C{c} sequences are bounded, we use a stronger condition \ref{g3}.
\begin{lemma} \label{Lemma 7}
If \ref{g1}--\ref{g3} hold, then every {\em \C{c}} sequence $\seq{u_j} \subset X^p_0(\Omega)$ is bounded for all $c \in \R$.
\end{lemma}

\begin{proof}
We have $\Phi(u_j) = c + \o(1)$ and $(1 + \norm{u_j}) \dualp{\Phi'(u_j)}{u_j} = \o(\norm{u_j})$, so
\begin{equation} \label{12}
\frac{1}{p}\, \E_{L,\,p}(u_j,u_j) - \int_\Omega G(u_j)\, dx = c + \o(1)
\end{equation}
and
\begin{equation} \label{13}
\E_{L,\,p}(u_j,u_j) - \int_\Omega g(u_j)\, u_j\, dx = \o(1).
\end{equation}
Combining \eqref{13} with Lemma \ref{Lemma 1} and \eqref{6} shows that for any $\eps > 0$, there exists a constant $C_\eps > 0$ such that
\begin{equation} \label{17}
\norm{u_j}^p \le \eps \int_\Omega |u_j|^p \abs{\ln |u_j|} dx + C_\eps \pnorm{u_j}^p + \o(1).
\end{equation}
Combining \eqref{12} and \eqref{13} with \ref{g3} gives
\begin{equation} \label{14}
\int_{\set{|u_j| \ge t_0}} \frac{|u_j|^p}{\ln |u_j|}\, dx \le C_2
\end{equation}
for some constant $C_2 > 0$. By the H\"{o}lder inequality,
\[
\int_{\set{|u_j| \ge t_0}} |u_j|^p\, dx \le \left(\int_{\set{|u_j| \ge t_0}} \frac{|u_j|^p}{\ln |u_j|}\, dx\right)^{1/2} \left(\int_{\set{|u_j| \ge t_0}} |u_j|^p \ln |u_j|\, dx\right)^{1/2}.
\]
Combining the last inequality with \eqref{14} and Corollary \ref{Corollary 1} gives
\[
\pnorm{u_j}^p \le \widetilde{C}_\delta\, \big(\norm{u_j}^{p/2} + \pnorm{u_j}^{(p + \delta)/2} + 1\big)
\]
for some constant $\widetilde{C}_\delta > 0$. Taking $\delta$ so small that $(p + \delta)/2 < p$ in this inequality gives
\begin{equation} \label{18}
\pnorm{u_j}^p \le C_3 \left(\norm{u_j}^{p/2} + 1\right)
\end{equation}
for some constant $C_3 > 0$. Using this estimate in \eqref{20} now gives
\begin{equation} \label{19}
\int_\Omega |u_j|^p \abs{\ln |u_j|} dx \le C_4\, \big(\norm{u_j}^p + \norm{u_j}^{(p + \delta)/2} + 1\big) \le C_5 \left(\norm{u_j}^p + 1\right)
\end{equation}
for some constants $C_4, C_5 > 0$. Using \eqref{18} and \eqref{19} in \eqref{17} gives
\[
\norm{u_j}^p \le \eps\, C_5 \norm{u_j}^p + \widetilde{C}_\eps \left(\norm{u_j}^{p/2} + 1\right)
\]
for some constant $\widetilde{C}_\eps > 0$. Since $\eps > 0$ is arbitrary, we obtain $\norm{u_j}$ is bounded.
\end{proof}


\begin{lemma} \label{Lemma 9}
If \ref{g1}--\ref{g3} hold, then $\Phi$ satisfies the {\em \C{c}} condition for all $c \in \R$.
\end{lemma}
\begin{proof}
    The proof follows from Lemma \ref{Lemma 6} and Lemma \ref{Lemma 7}.
\end{proof}

Next, we will obtain nontrivial solutions of problem \eqref{5} using a variant of a linking theorem proved in Yang and Perera \cite{MR3616328}. We begin by recalling the definition of linking.

\begin{definition}
Let $A$ and $B$ be nonempty closed subsets of a Banach space $W$ such that $A$ is bounded and $\dist(A,B) > 0$. Let
\begin{equation} \label{28}
X = \set{tu : u \in A,\, t \in [0,1]}, \quad \Gamma = \set{\gamma \in C(X,W) : \gamma(X) \text{ is closed and } \restr{\gamma}{A} = \id[A]}.
\end{equation}
We say that $A$ links $B$ if
\[
\gamma(X) \cap B \ne \emptyset \quad \forall \gamma \in \Gamma.
\]
\end{definition}

The following proposition is standard (see, e.g., Perera et al.\! \cite[Proposition 3.21]{MR2640827}).

\begin{proposition} \label{Proposition 1}
Let $\Phi$ be a $C^1$-functional on $W$. Let $A$ and $B$ be nonempty closed subsets of $W$ such that $A$ is bounded and $\dist(A,B) > 0$. Assume that $A$ links $B$ and
\[
\sup_{u \in A}\, \Phi(u) \le \inf_{u \in B}\, \Phi(u).
\]
Let $X$ and $\Gamma$ be as in \eqref{28}. Assume that
\[
c := \inf_{\gamma \in \Gamma}\, \sup_{u \in \gamma(X)}\, \Phi(u)
\]
is finite and $\Phi$ satisfies the {\em \C{c}} condition. Then $c \ge \inf \Phi(B)$ is a critical value of $\Phi$. If $c = \inf \Phi(B)$, then $\Phi$ has a critical point with critical value $c$ on $B$. In particular, if $\inf \Phi(B) \ge \Phi(0)$ in addition, then $\Phi$ has a nontrivial critical point with critical value $c$.
\end{proposition}

To state the linking theorem that we will be using, let $\M$ be a closed symmetric subset of $W \setminus \set{0}$ that is radially homeomorphic to the unit sphere $S = \set{u \in W : \norm{u} = 1}$, i.e., the restriction to $\M$ of the radial projection $\pi_S : W \setminus \set{0} \to S,\, u \mapsto u/\norm{u}$ is a homeomorphism. Then the radial projection on $\M$ is given by
\[
\pi_\M = (\restr{\pi_S}{\M})^{-1} \comp \pi_S.
\]
Recall that for a symmetric set $A \subset W \setminus \set{0}$, $i(A)$ denotes its $\Z_2$-cohomological index. We have the following theorem.

\begin{theorem} \label{Theorem 3}
Let $\Phi$ be a $C^1$-functional on $W$. Let $A_0$ and $B_0$ be disjoint nonempty symmetric subsets of $\M$ such that $A_0$ is compact, $B_0$ is closed, and
\begin{equation} \label{29}
i(A_0) = i(\M \setminus B_0) < \infty.
\end{equation}
Assume that there exist $u_0 \in \M \setminus A_0$ and $R, \rho > 0$ such that $\rho < R\, \dist(0,\M)$ and
\begin{equation} \label{36}
\sup_{u \in A}\, \Phi(u) \le \inf_{u \in B}\, \Phi(u),
\end{equation}
where $A = \set{R\, tu : u \in A_0,\, t \in [0,1]} \cup \set{R\, \pi_\M((1 - t)u + tu_0) : u \in A_0,\, t \in [0,1]}$ and $B = \set{\rho\, \pi_S(u) : u \in B_0}$. Let $X$ and $\Gamma$ be as in \eqref{28}. Assume that
\[
c := \inf_{\gamma \in \Gamma}\, \sup_{u \in \gamma(X)}\, \Phi(u)
\]
is finite and $\Phi$ satisfies the {\em \C{c}} condition. Then $\Phi$ has a nontrivial critical point with critical value $c$.
\end{theorem}

\begin{proof}
We apply Proposition \ref{Proposition 1}. Since $A_0$ is compact and $\pi_\M$ is continuous, $A$ is compact and hence closed and bounded. Since $B_0$ is closed and $\restr{\pi_S}{\M}$ is a homeomorphism, $B$ is also closed. Since $A_0$ and $B_0$ are disjoint, so are $\set{R\, tu : u \in A_0,\, t \in [0,1]}$ and $B$. Since $\rho < R\, \dist(0,\M)$, $\set{R\, \pi_\M((1 - t)u + tu_0) : u \in A_0,\, t \in [0,1]}$ and $B$ are also disjoint. So $A$ and $B$ are disjoint. Then $\dist(A,B) > 0$ since $A$ is compact and $B$ is closed.

It only remains to show that $A$ links $B$. Let $\Sigma A_0$ be the suspension of $A_0$. By Proposition \ref{Proposition 7} $(i_6)$,
\begin{equation} \label{31}
i(\Sigma A_0) = i(A_0) + 1.
\end{equation}
Let
\[
\widetilde{A} = \set{R\, \pi_\M((1 - t)u + tu_0) : u \in A_0,\, t \in [0,1]} \cup \set{R\, \pi_\M((1 - t)u - tu_0) : u \in A_0,\, t \in [0,1]}
\]
and note that $\widetilde{A}$ is closed since $A_0$ is compact and $\pi_\M$ is continuous (here $(1 - t)u \pm tu_0 \ne 0$ since $\pm u_0 \notin A_0$ and each ray starting from the origin intersects $\M$ at exactly one point). We have the odd continuous map
\[
\Sigma A_0 \to \widetilde{A}, \quad (u,t) \mapsto \begin{cases}
R\, \pi_\M((1 - t)u + tu_0) & \text{if } (u,t) \in A_0 \times [0,1]\\[5pt]
R\, \pi_\M((1 + t)u + tu_0) & \text{if } (u,t) \in A_0 \times [-1,0),
\end{cases}
\]
so
\begin{equation}
i(\Sigma A_0) \le i(\widetilde{A})
\end{equation}
by Proposition \ref{Proposition 7} $(i_2)$.

Suppose $A$ does not link $B$. Then there exists $\gamma \in \Gamma$ such that $\gamma(X) \cap B = \emptyset$, where $X$ and $\Gamma$ are as in \eqref{28}. Consider the map $\varphi : \widetilde{A} \times [0,1] \to W$ defined by
\[
\varphi(u,t) = \begin{cases}
\gamma(tu) & \text{if } (u,t) \in (\widetilde{A} \cap A) \times [0,1]\\[5pt]
- \gamma(-tu) & \text{if } (u,t) \in (\widetilde{A} \setminus A) \times [0,1].
\end{cases}
\]
Since $\gamma$ is the identity on $\set{R\, tu : u \in A_0,\, t \in [0,1]}$, $\varphi$ is continuous. Clearly, $\varphi(-u,t) = - \varphi(u,t)$ for all $(u,t) \in \widetilde{A} \times [0,1]$. Since $\gamma(X)$ is closed, so is $\varphi(\widetilde{A} \times [0,1]) = \gamma(X) \cup -\gamma(X)$. Since $\restr{\gamma}{A} = \id[A]$, $\varphi(\widetilde{A} \times \set{0}) = \set{0}$ and $\varphi(\widetilde{A} \times \set{1}) = \widetilde{A}$. Noting that
\[
\norm{u} \ge R\, \dist(0,\M) > \rho \quad \forall u \in \widetilde{A}
\]
and applying Proposition \ref{Proposition 7} $(i_7)$ with $C = \widetilde{A}$, $C_0 = \set{u \in W : \norm{u} \le \rho}$, and $C_1 = \{u \in W : \norm{u} \ge \rho\}$ gives
\begin{equation}
i(\widetilde{A}) \le i(\varphi(\widetilde{A} \times [0,1]) \cap S_\rho),
\end{equation}
where $S_\rho = \set{u \in W : \norm{u} = \rho}$. Since $\gamma(X) \subset W \setminus B$ and $W \setminus B$ is symmetric,
\[
\varphi(\widetilde{A} \times [0,1]) = \gamma(X) \cup -\gamma(X) \subset W \setminus B.
\]
So
\[
\varphi(\widetilde{A} \times [0,1]) \cap S_\rho \subset (W \setminus B) \cap S_\rho = S_\rho \setminus B
\]
and hence
\begin{equation}
i(\varphi(\widetilde{A} \times [0,1]) \cap S_\rho) \le i(S_\rho \setminus B)
\end{equation}
by Proposition \ref{Proposition 7} $(i_2)$. Since the restriction of $\pi_\M$ to $S_\rho \setminus B$ is an odd homeomorphism onto $\M \setminus B_0$, Proposition \ref{Proposition 7} $(i_2)$ also gives
\begin{equation} \label{30}
i(S_\rho \setminus B) = i(\M \setminus B_0).
\end{equation}
Combining \eqref{31}--\eqref{30} gives $i(A_0) + 1 \le i(\M \setminus B_0)$, contradicting \eqref{29}.
\end{proof}

Theorem \ref{Theorem 3} holds when $A_0 = \emptyset$ and $B_0 = \M$ also if we take $A = \set{0,Ru_0}$. It reduces to the following version of the mountain pass theorem in this case.

\begin{theorem} \label{Theorem 4}
Let $\Phi$ be a $C^1$-functional on $W$. Assume that there exist $u_1 \in W \setminus \set{0}$ and $\rho > 0$ such that $\rho < \norm{u_1}$ and
\begin{equation} \label{34}
\max \set{\Phi(0),\Phi(u_1)} \le \inf_{u \in S_\rho}\, \Phi(u),
\end{equation}
where $S_\rho = \set{u \in W : \norm{u} = \rho}$. Let $\Gamma = \set{\gamma \in C([0,1],W) : \gamma(0) = 0 \text{ and } \gamma(1) = u_1}$. Assume that
\[
c := \inf_{\gamma \in \Gamma}\, \max_{u \in \gamma([0,1])}\, \Phi(u)
\]
is finite and $\Phi$ satisfies the {\em \C{c}} condition. Then $\Phi$ has a nontrivial critical point with critical value $c$.
\end{theorem}

In order to apply Theorem \ref{Theorem 3} and Theorem \ref{Theorem 4} to the functional $\Phi$ in \eqref{32}, first we determine its asymptotic behavior near the origin.

\begin{lemma} \label{Lemma 8}
If \ref{g1} and \ref{g2} hold, then
\[
\Phi(u) = \frac{1}{p}\, \E_{L,\,p}(u,u) - \frac{\lambda}{p} \int_\Omega |u|^p\, dx + \o(\norm{u}^p) \quad \text{as } \norm{u} \to 0.
\]
\end{lemma}

\begin{proof}
For any $\eps > 0$, there exists $\delta > 0$ such that
\begin{equation} \label{21}
\abs{g(t) - \lambda\, |t|^{p-2}\, t} \le \eps\, |t|^{p-1} \quad \text{for } |t| < \delta
\end{equation}
by \ref{g1}. By \ref{g2},
\[
\frac{g(t) - \lambda\, |t|^{p-2}\, t}{|t|^{p-2}\, t \ln |t|} \to 0 \quad \text{as } |t| \to \infty,
\]
so there exists $M > 1$ such that
\begin{equation}
\abs{g(t) - \lambda\, |t|^{p-2}\, t} \le \eps\, |t|^{p-1} \ln |t| \quad \text{for } |t| > M.
\end{equation}
Fix $\gamma \in (0,1)$ and let
\[
h(t) = \frac{d}{dt}\, \big[|t|^p\, (\ln (1 + |t|))^\gamma\big] = \begin{cases}
p\, |t|^{p-2}\, t\, (\ln (1 + |t|))^\gamma + \dfrac{\gamma\, |t|^{p-1}\, t}{(1 + |t|)(\ln (1 + |t|))^{1 - \gamma}} & \text{for } t \ne 0\\[15pt]
0 & \text{for } t = 0.
\end{cases}
\]
Then $h$ is continuous and nonzero for $t \ne 0$, so
\begin{equation} \label{22}
\abs{g(t) - \lambda\, |t|^{p-2}\, t} \le C_\eps\, |h(t)| \quad \text{for } \delta \le |t| \le M,
\end{equation}
where
\[
C_\eps = \max_{\delta \le |t| \le M}\, \abs{\frac{g(t) - \lambda\, |t|^{p-2}\, t}{h(t)}}.
\]

Combining \eqref{21}--\eqref{22} gives
\[
\abs{g(t) - \lambda\, |t|^{p-2}\, t} \le \eps\, |t|^{p-1} + C_\eps\, |h(t)| + \eps\, \goodchi_{\set{|t| > 1}}\, |t|^{p-1} \ln |t| \quad \forall t \in \R,
\]
which implies that
\[
\abs{G(t) - \frac{\lambda}{p}\, |t|^p} \le \eps\, |t|^p + C_\eps\, |t|^p\, (\ln (1 + |t|))^\gamma + \eps\, \goodchi_{\set{|t| > 1}}\, |t|^p \ln |t| \quad \forall t \in \R.
\]
This gives
\[
\abs{\int_\Omega G(u)\, dx - \frac{\lambda}{p} \int_\Omega |u|^p\, dx} \le \eps \int_\Omega |u|^p\, dx + C_\eps \int_\Omega |u|^p\, (\ln (1 + |u|))^\gamma\, dx + \eps \int_{\set{|u| > 1}} |u|^p \ln |u|\, dx.
\]
Combining this with Corollary \ref{Corollary 2} and Corollary \ref{Corollary 3} gives the desired conclusion since $\eps > 0$ is arbitrary.
\end{proof}

Let $\M := \M_p$ and $\I = \I_{p,p}$ where $\M_p$ and $\I_{p,p}$ are defined in \eqref{26}. Combining Lemma \ref{Lemma 8} with Lemma \ref{Lemma 1} gives the following estimate for $\Phi$.

\begin{lemma} \label{Lemma 10}
If \ref{g1} and \ref{g2} hold, then for any $\eps > 0$, there exists $\rho_0 > 0$ such that for all $\rho \in (0,\rho_0)$,
\begin{equation} \label{33}
\Phi\bigg(\frac{\rho u}{\norm{u}}\bigg) \ge \bigg(\frac{\rho}{\norm{u}}\bigg)^{\!p} \big[(1 - \eps)\, \I(u) - (\lambda + \eps)\big] \quad \forall u \in \M.
\end{equation}
\end{lemma}

\begin{proof}
Let $C > 0$ be the constant in Lemma \ref{Lemma 1} and set
\[
\widetilde{\eps} = \frac{\eps}{2p\, \max \set{C,1}}.
\]
By Lemma \ref{Lemma 8}, there exists $\rho_0 > 0$ such that for all $\rho \in (0,\rho_0)$ and $u \in X^p_0(\Omega)$ with $\norm{u} \le \rho$,
\[
\Phi(u) \ge \frac{1}{p}\, \E_{L,\,p}(u,u) - \frac{\lambda}{p} \int_\Omega |u|^p\, dx - \widetilde{\eps}\, \norm{u}^p.
\]
Combining this with Lemma \ref{Lemma 1} gives
\[
\Phi(u) \ge \frac{1}{p}\, (1 - \eps)\, \E_{L,\,p}(u,u) - \frac{1}{p}\, (\lambda + \eps) \int_\Omega |u|^p\, dx = (1 - \eps)\, I_p(u) - (\lambda + \eps)\, J_p(u)
\]
for all such $\rho$ and $u$. This, in turn, gives
\[
\Phi\bigg(\frac{\rho u}{\norm{u}}\bigg) \ge \bigg(\frac{\rho}{\norm{u}}\bigg)^{\!p} \big[(1 - \eps)\, I_p(u) - (\lambda + \eps)\, J_p(u)\big]
\]
for all $\rho \in (0,\rho_0)$ and $u \in X^p_0(\Omega) \setminus \set{0}$. For $u \in \M$, this reduces to \eqref{33}.
\end{proof}

Next we show that $\Phi(tu) \to - \infty$ as $t \to \infty$, uniformly on sublevel sets of $\I$. It follows from \ref{g3} that
\begin{equation*} \label{11}
\dfrac{G(t)}{|t|^p} \ge \begin{cases}
\dfrac{G(t_0)}{t_0^p}\, + \beta \ln \left(\dfrac{\ln t}{\ln t_0}\right) & \text{for } t \ge t_0\\[15pt]
\dfrac{G(-t_0)}{t_0^p}\, + \beta \ln \left(\dfrac{\ln |t|}{\ln t_0}\right) & \text{for } t \le -t_0,
\end{cases}
\end{equation*}
which further implies
\begin{equation}\label{super-linear}
    \lim\limits_{|t| \to \infty} \frac{G(t)}{|t|^p} = +\infty.
\end{equation}

\begin{lemma} \label{Lemma 11}
If \ref{g3} holds, then for any $a \in \R$, there exists a constant $C_a > 0$ such that
\[
\Phi(tu) \le t^p\, (\I(u) - a) + C_a \quad \forall u \in \M,\, t \ge 0.
\]
\end{lemma}

\begin{proof}
By \eqref{super-linear}, there exists $M \ge t_0$ such that
\[
G(t) \ge \frac{a}{p}\, |t|^p \quad \text{for } |t| \ge M.
\]
So
\[
G(t) \ge \frac{a}{p}\, |t|^p - C \quad \forall t \in \R
\]
for some constant $C > 0$, and the desired inequality follows.
\end{proof}

We are now ready to obtain nontrivial solutions of problem \eqref{5}. First we obtain a nontrivial solution when $\lambda < \lambda_1$, where $\lambda_1$ is the first Dirichlet eigenvalue of $L_{\Delta_p}$ given by \eqref{3}, using Theorem \ref{Theorem 4}.

\begin{theorem}\label{thm:existence:lambda1}
If \ref{g1}--\ref{g3} hold with $\lambda < \lambda_1$, then problem \eqref{5} has a nontrivial solution.
\end{theorem}

\begin{proof}
Take any $u_0 \in \M$. Applying Lemma \ref{Lemma 11} with $a > \I(u_0)$ gives $R > 0$ such that $u_1 = Ru_0$ satisfies $\Phi(u_1) \le 0$. Noting that $\I(u) \ge \lambda_1$ for all $u \in \M$ by \eqref{3} and applying Lemma \ref{Lemma 10} with
\[
\eps = \frac{\lambda_1 - \lambda}{\lambda_1 + 1} \quad \text{if} \ \lambda_1 > -1 \quad \text{and} \quad \eps > 0 \quad \text{if} \ \lambda_1 \leq -1,
\]
gives $\rho > 0$ such that $\rho < \norm{u_1}$ and $\inf \Phi(S_\rho) \ge 0$. So \eqref{34} holds. Since $\Phi$ satisfies the \C{c} condition for all $c \in \R$ by Lemma \ref{Lemma 9}, then $\Phi$ has a nontrivial critical point by Theorem \ref{Theorem 4}.
\end{proof}

When $\lambda > \lambda_1$, we use Theorem \ref{Theorem 3} to obtain a nontrivial solution of problem \eqref{5} provided that $\lambda$ is not an eigenvalue of $L_{\Delta_p}$ from the sequence $\seq{\lambda_k}$ given in Theorem \ref{Theorem 1}.

\begin{theorem}\label{thm:existence:lambdak}
If \ref{g1}--\ref{g3} hold with $\lambda_k < \lambda < \lambda_{k+1}$ and
\begin{equation} \label{35}
G(t) \ge \frac{\widetilde{\lambda}}{p}\, |t|^p \quad \forall t \in \R
\end{equation}
for some $\widetilde{\lambda} \in (\lambda_k,\lambda]$, then problem \eqref{5} has a nontrivial solution.
\end{theorem}

\begin{proof}
Since $\lambda_k < \widetilde{\lambda} < \lambda_{k+1}$,
\[
i(\M \setminus \I_{\widetilde{\lambda}}) = i(\M \setminus \I_{\lambda_{k+1}}) = k
\]
by Theorem \ref{Theorem 1} \ref{Theorem 1.iii}. Since $\M \setminus \I_{\widetilde{\lambda}}$ is an open symmetric subset of $\M$, it has a compact symmetric subset $A_0$ with the same index (see Perera et al.\! \cite[Proposition 2.14 $(iii)$]{MR2640827}). Let $B_0 = \I_{\lambda_{k+1}}$ and note that \eqref{29} holds.

Take any $u_0 \in \M \setminus A_0$ and let $A$ and $B$ be as in Theorem \ref{Theorem 3}. By \eqref{35},
\[
\Phi(tu) \le t^p\, (\I(u) - \widetilde{\lambda}) \le 0 \quad \forall u \in A_0,\, t \ge 0.
\]
Since $A_0$ is compact and $\pi_\M$ is continuous, the set $A_1 = \set{\pi_\M((1 - t)u + tu_0) : u \in A_0,\, t \in [0,1]}$ is compact and hence $\I$ is bounded on $A_1$. Applying Lemma \ref{Lemma 11} with $a > \sup \I(A_1)$ gives $R > 0$ such that $A_R = \set{R\, \pi_\M((1 - t)u + tu_0) : u \in A_0,\, t \in [0,1]}$ satisfies $\sup \Phi(A_R) \le 0$. Since $\I(u) \ge \lambda_{k+1}$ for all $u \in B_0$, applying Lemma \ref{Lemma 10} with
\[
\eps = \frac{\lambda_{k+1} - \lambda}{\lambda_{k+1} + 1} \quad \text{if} \ \lambda_{k+1} > -1 \quad \text{and} \quad \eps > 0 \quad \text{if} \ \lambda_{k+1} \leq -1,
\]
gives $\rho > 0$ such that $\rho < R\, \dist(0,\M)$ and $\inf \Phi(B) \ge 0$. So \eqref{36} holds. Since $\Phi$ satisfies the \C{c} condition for all $c \in \R$ by Lemma \ref{Lemma 9}, then $\Phi$ has a nontrivial critical point by Theorem \ref{Theorem 3}.
\end{proof}
\def\cprime{$''$}

\end{document}